\documentclass[a4paper,11pt]{amsart}
\usepackage{amssymb}
\usepackage{cmap}
\usepackage{hyperxmp}
\usepackage[pdfdisplaydoctitle=true,
            colorlinks=true,
            urlcolor=blue,
            citecolor=blue,
            linkcolor=blue,
            pdfstartview=FitH,
            pdfpagemode= UseNone,
            bookmarksnumbered=true,
            backref=page]{hyperref}
\usepackage{todonotes}
\usepackage{lipsum}

\theoremstyle{plain}
\newtheorem{thm}{Theorem}[subsection]
\newtheorem*{thm*}{Theorem}

\newtheorem{cor}[thm]{Corollary}

\newtheorem{prop}[thm]{Proposition}

\theoremstyle{definition}

\theoremstyle{definition}
\newtheorem{rem}[thm]{Remark}

\newtheorem*{ass*}{Assumption}

\numberwithin{equation}{section}

\newcommand\blfootnote[1]{%
  \begingroup
  \renewcommand\thefootnote{}\footnote{#1}%
  \addtocounter{footnote}{-1}%
  \endgroup
}

\newcommand{\RR}{\mathbb{R}} %
\newcommand{\id}{\mathrm{id}}           %
\let\on=\operatorname

\newcommand{\Diff}{\mathrm{Diff}}       %
\newcommand{\biLip}{\mathrm{biLip}}

\def\vp{\varphi}

\let\on=\operatorname

\def\Diff{\on{Diff}}

\allowdisplaybreaks

\newcommand{\frakg}{\mathfrak{g}}

\hypersetup{
    pdftitle={Sobolev metrics on spaces of manifold valued curves}, %
    pdfsubject={MSC 2010: 58B20, 58D10, 35G55, 35G60}, %
    pdfkeywords={Sobolev metrics, manifold valued curves}, %
    pdflang=EN %
    }
\author{Martin Bauer\textsuperscript{*}}
\address{Martin Bauer: Department of Mathematics, Florida State University}
\email{bauer@math.fsu.edu}
\author{Yuxiu Lu }
\address{Yuxiu Lu: Department of Mathematics, Florida State University}
\email{yl18f@my.fsu.edu }

\author{Cy Maor}
\address{Cy Maor: Einstein Institute of Mathematics, The Hebrew University of Jerusalem}
\email{cy.maor@mail.huji.ac.il}
\thanks{\noindent * Corresponding author\\ M.~Bauer was partially supported by NSF-grants 1912037 and 
1953244. Y.~Lu was partially supported by NSF-grant 1912037. 
C.~Maor was partially supported by ISF-grant 1269/19.}

\makeatletter
\@namedef{subjclassname@2020}{\textup{2020} Mathematics Subject Classification}
\makeatother

\begin{document}

\title[Generalized Proudman--Johnson and $r$-Hunter--Saxton equations]{A geometric view on the generalized Proudman--Johnson and $r$-Hunter--Saxton equations}

%
%

%
\begin{abstract}
We show that two families of equations on the real line, the generalized inviscid Proudman--Johnson equation, and the $r$-Hunter--Saxton equation (recently introduced by Cotter et al.) coincide for a certain range of parameters.
This gives a new geometric interpretation of these Proudman--Johnson equations as geodesic equations of right invariant homogeneous $W^{1,r}$-Finsler metrics on an appropriate diffeomorphism group on $\mathbb{R}$.
Generalizing a construction of Lenells for the Hunter--Saxton equation, we analyze the $r$-Hunter--Saxton equation using an isometry from the diffeomorphism group to an appropriate subset of real-valued functions. Thereby we show that the periodic case is equivalent to the geodesic equation on the $L^r$-sphere in the space of functions, and the non-periodic case is equivalent to a geodesic flow on a flat space.
This allows us to give explicit solutions to these equations in the non-periodic case, and answer several questions of Cotter et al.\ regarding their limiting behavior.
\end{abstract}

\maketitle

\setcounter{tocdepth}{1}
\tableofcontents

\blfootnote{This is a corrected version of the published version~\cite{bauer2022geometric}. 
In it, we study two families of equations: the generalized inviscid Proudman--Johnson equation, and the $r$-Hunter--Saxton equation, both on the real line and the circle. In the published version we claim that they are equivalent in both cases.  
This statement is wrong, as the equations are only equivalent when considered on the real line. 
This updated version is the corrected one. }

\section{Introduction and main results}
In his seminal article~\cite{arnold1966geometrie} Arnold found a geometric interpretation of the incompressible Euler equation as the geodesic equation of a right-invariant Riemannian metric on the group of diffeomorphisms. Since then an analogues geometric picture has been constructed for several other equations in mathematical hydrodynamics and 
equations that admit such an interpretation are referred to as Euler--Arnold equations. Examples include the Camassa--Holm~\cite{camassa1993integrable,misiolek1998shallow,kouranbaeva1999camassa}, the Hunter--Saxton~\cite{hunter1991dynamics,lenells2007hunter}, the Burgers, the KdV~\cite{ovsienko1987korteweg}, or the modified Constantin--Lax--Majda equation \cite{constantin1985simple,wunsch2010geodesic,escher11geometry}, see also~\cite{vizman2008geodesic,arnold1999topological,bauer2014overview} for further examples. 

In this article we study two families of equations.
The first one is the $r$-Hunter--Saxton equation (henceforth $r$-HS), recently introduced by by Cotter et al.~\cite{cotter2020r}:
\[
 \left(u_x |u_x|^{r-2}\right)_{xt}+\left(u_x |u_x|^{r-2}\right)_{x}u_x+\left((u_x |u_x|^{r-2})_{x}u\right)_x=0.
\]
This one-parameter family of PDEs is derived as the geodesic equation, in Eulerian coordinates, of the right-invariant, homogenous $W^{1,r}$--Finsler metric (henceforth the $\dot W^{1,r}$ metric). To emphasize the relation to Euler--Arnold equations, which are defined as geodesic equations of a right-invariant Riemannian metrics, we will refer such equations as Finsler--Euler--Arnold equations.

The second family is the inviscid, generalized Proudman--Johnson equations with parameter $\lambda$ (henceforth $\lambda$-PJ equation), which are given by
\begin{equation}
u_{txx} + (1+2\lambda)u_xu_{xx} +uu_{xxx} = 0, 
\end{equation}
where $u$ is either a function on the real line (non-periodic $\lambda$-PJ) or on the circle (periodic $\lambda$-PJ).
The original Proudman--Johnson equation, in which $\lambda = -1$, corresponds to axisymmetric Navier--Stokes equations in $\RR^2$; 
the generalized equations, which were first proposed in \cite{okamoto2000some}, contain several other important special cases, in particular the Hunter--Saxton equation for $\lambda=1/2$, the $\mu$-Burgers equation $\lambda=1$, and self-similar axisymmetric Navier--Stokes equations in higher dimensions.
See \cite{Wun11} for further information about the equation and its motivation.

In the article~\cite{LM14} Lenells and Misio{\l}ek constructed a geometric interpretation of the $\lambda$-PJ equations for $\lambda \in [0,1]$ by interpreting them as geodesic equations of an affine connection $\nabla^{\alpha}$ on the homogenous space of all diffeomorphism of the circle modulo the group of rotations (in their notation $\alpha = 1-2\lambda$).\footnote{See also \cite{escher2011degasperis} for a similar interpretation of the $b$-equations, which include the Camassa--Holm and the Degasperis--Procesi equation.} 
The connection reduces to a Levi-Civita connection of the homogenous $W^{1,2}$ Riemannian metric for $\lambda=1/2$. In this case this recasts the geometric interpretation of the Hunter--Saxton equation as an Euler--Arnold equation, as described in \cite{lenells2007hunter,lenells2008hunter}.

\subsection*{Contributions of the article}

The starting point of our investigations is the observation that in the non-periodic case (i.e., on $\mathbb{R}$) the family of $r$-HS equations is formally equivalent to the $\lambda$-PJ equations for any $\lambda\in (0,1)$, with $\lambda = 1/r$. Using the derivation of the $r$-HS equation as geodesic equation of a right-invariant Finsler metric, one has thus found an interpretation of the $\lambda$-PJ equations as Finsler--Euler--Arnold equations and thereby complemented the geometric picture of Lenells and Misio{\l}ek~\cite{LM14}. This observation positively answers a recent question by Gibilisco as proposed in~\cite[Problem 5]{Gib20} in the non-periodic case.

Using this geometric picture, as geodesic equation of the right-invariant, $\dot W^{1,r}$-metric, allows one to investigate the properties of the $\lambda$-PJ equations by studying the geometry of the corresponding infinite dimensional group. 
In particular, we extend a construction, originally found by Lenells for the periodic Hunter--Saxton equation~\cite{lenells2007hunter}, and later extended to the non-periodic case in \cite{bauer2014homogeneous}, to the whole family of  $\dot W^{1,r}$-metric, for $1\leq r<\infty$.
This construction, which isometrically maps the diffeomorphism group to a (subset of a) vector space of functions, implies the following: 
\begin{itemize}
\item The non-periodic-$r$-HS equation (and correspondingly $\lambda$-PJ) is equivalent to geodesics with respect to the $L^r$ norm in an open and convex subset of $W^{\infty,1}(\mathbb R)={\bigcap}_{k\in \mathbb{N}} W^{k,1}(\mathbb R)$ (Section~\ref{sec:transformation_non_per}).
\item The periodic-$r$-HS equation is equivalent to geodesics on the $L^r$-sphere on $C^\infty(S^1)$ (Section~\ref{sec:transformation_per}).
\end{itemize}
In this article we mainly focus in the non-periodic case, where 
this equivalence turns out to be utmost advantageous for studying the equations. Since $W^{\infty,1}(\RR)$ (with the $L^r$-norm) is a vector space, geodesics in it are given by straight lines; from this we obtain explicit formulas for the solutions of the non-periodic-$\lambda$-PJ equation for any $\lambda\in (0,1)$, see Section~\ref{sec:sol_formula}.
While this discussion is done in smooth settings, we present in Section~\ref{sec:lowregularity_nonperiodic}, using a Lagrangian approach, a functional setting for discussing these equations in low regularity, namely for integrable, Lipschitz velocities $u$. 
In particular, this gives a rigorous interpretation of the piecewise-linear solutions studied in Cotter et al.~\cite{cotter2020r}, and shows that the space of piecewise-linear velocities is totally geodesic.
We use this geometric approach to easily retrieve some of the results of Cotter et al.~\cite{cotter2020r}, such as blow-up time of solutions (Section~\ref{sec:blowup}), as well as answer several questions stated in \cite{cotter2020r} regarding the limiting behavior of solutions as $r\to \infty$ ($\lambda\to 0$, resp.): we show that the solutions converge, as $r\to \infty$, to a solution of the $\lambda$-PJ equation for $\lambda=0$, and that this solution preserves the $\dot W^{1,\infty}$-norm of the velocity along the flow (Section~\ref{sec:r_to_infty}).
We also show existence of the boundary-value problem in Lagrangian coordinates (Section~\ref{sec:bvp}).
Finally, in Section~\ref{sec:r<1}, we show that large parts of our analysis continue to hold for $\lambda \notin [0,1]$ ($r<1$, resp.): while the geometric interpretation of the PDEs as Finsler--Euler--Arnold equations breaks down, the solution formula is still valid for this wider range of parameters.

The periodic case is considered in Section~\ref{sec:periodic}.
There, the transformation to the $L^r$-sphere does not immediately yields explicit solutions (when $r\ne 2$), but, similar to the non-periodic case, it provides a framework for discussing these equations in lower regularity (Lipschitz velocities).
A similar transformation for the periodic $\lambda$-PJ equation was already used in the analysis of blow-up of the solutions by Sarria and Saxton \cite{SS13}. 
Our contribution unveils the geometry behind this transformation for the $r$-HS euqation.
In a recent paper by Kogelbauer~\cite{Kog20}, explicit formulae for the flow map of the $\lambda$-PJ equation, up to the solution of an ODE, were obtained; it will be interesting to relate these solutions to the geometric picture on the sphere.

We note that for the non-periodic case our construction directly provides a change of coordinates that linearizes the flow of the $\lambda$-PJ equation. 
In the spirit of \cite{LM14}, this can be interpreted as an analogue of the inverse scattering transform formalism, and thus the integrability of these equations for any $\lambda\in (0,1)$ follows, and in fact to any $\lambda \in \RR$ in view of the results of Section~\ref{sec:r<1}.
To the best of our knowledge, integrability of the $\lambda$-PJ equations was previously only known for the case $\lambda=1/2$ (the Hunter--Saxton equation), $\lambda=1$ (the Burger's equation) and $\lambda =0$.\footnote{For $\lambda=0$ the integrability of the periodic $\lambda$-PJ equation is shown in the article~\cite{LM14}. This proof translates directly to the non-periodic case.}
Our results therefore confirm the analogue of~\cite[Problem 6]{Gib20} for the non-periodic $\lambda$-PJ equation.

\subsection*{Acknowledgements}
The authors are grateful to S.\ Preston 
and G.\ Misio{\l}ek 
for various discussions during the preparation of the manuscript.

\section{The non-periodic case}
In this section we will study the geometric picture for the non-periodic $r$-Hunter--Saxton equation (generalized Proudman--Johnson equation resp.), as the Finsler--Euler--Arnold equation on a group of diffeomorphisms equipped with a right-invariant Finsler metric. 
The unboundedness of $\mathbb R$ will require us to specify appropriate decay conditions for the elements of the diffeomorphisms group, which we introduce below in Section~\ref{sec:diffR}. 
We then show that the $r$-HS equation is equivalent to the $\lambda$-PJ equation, for $\lambda= \frac{1}{r} \in (0,1)$.
We then proceed to the above-mentioned simplifying transformation, which
allows us to obtain explicit formulas for the solution of the geodesic equation, that in turn allows us to positively answer several questions and conjectures of Cotter et al.~\cite{cotter2020r} and Gibilisco~\cite{Gib20}.

\subsection{Diffeomorphisms on the circle and the $r$-HS ($\lambda$-PJ) equation}\label{sec:diffR}
We start by introducing an appropriate functional setting for studying the $r$-Hunter--Saxton equation on the real line. 
To this end, we need to identify a group of smooth, orientation preserving diffeomorphims on the real line, on which the flow will be defined.
Note that the group of all smooth, orientation preserving diffeomorphisms of the real line is not an open subset of the space $C^{\infty}(\mathbb R,\mathbb R)$ and consequently it is not a smooth Fr\'echet manifold. 
To overcome this difficulty one usually only consider diffeomorphisms that satisfy certain decay conditions, which then allows to retain a manifold structure for the corresponding space. 
While several different types of decay conditions have been considered in the literature~\cite{michor2013zoo,kriegl2015exotic}, we will restrict our analysis in this paper to groups related to the function space $W^{\infty,1}(\mathbb R)$. 
Specifically, we will consider the following diffeomorphism group:
\[
\Diff_{-\infty}(\mathbb R)=\left\{\varphi=\id+f: f'\in W^{\infty,1}(\mathbb R),\; f'>-1,\text{ and } \lim_{x\to -\infty} f(x)=0 \right\},
\]
where $W^{\infty,1}(\mathbb R)={\bigcap} W^{k,1}(\mathbb R)$ is defined as the intersection of all Sobolev spaces of order $k\geq 0$.
It has been shown in~\cite{bauer2014homogeneous} that this space is a smooth Fr\'echet Lie-groups with Lie-algebra:
\[
\mathfrak g_{-\infty}=\left\{u: u'\in W^{\infty,1}(\mathbb R)\text{ and } \lim_{x\to -\infty} u(x)=0 \right\}.
\]
The reason for working with this group and not the smaller, more commonly known one
\[
\begin{split}
&\Diff(\mathbb R) :=\left\{\varphi=\id+f: f\in W^{\infty,1}(\mathbb R)\text{ and } f'>-1 \right\} \\
	&\qquad =\left\{\varphi=\id+f: f'\in W^{\infty,1}(\mathbb R), f'>-1, \lim_{x\to \infty} f(x) =  \lim_{x\to -\infty} f(x)=0  \right\}
\end{split}
\]
will be clear from Theorem~\ref{thm:nonperiodicHS} below, where we show that the $r$-Hunter--Saxton equation is not consistent with two-sided decay conditions on $f$, and thus the geodesic equation is not even locally well-defined on $\Diff(\mathbb R)$.
Its well-definiteness on $\Diff_{-\infty}(\RR)$ will follow from Theorem~\ref{thm:geodesicsDiffR}.

We now define the right-invariant $\dot W^{1,r}$-Finsler metric.
To this end, we write any tangent vector $h\in T_\varphi \Diff_{-\infty}(\mathbb R)$ as $X\circ\varphi$ with 
$X\in \mathfrak g_{-\infty}$.
This allows us to 
define the right-invariant, homogenous, $W^{1,r}$-Finsler metric on 
$\Diff_{-\infty}(\mathbb R)$ via
\begin{equation}\label{eq:W1rnorm:R}
F_{\varphi}(h)=\left(\int_{\mathbb R} |(h\circ\varphi^{-1})'|^r dx\right)^{1/r}=
 \left(\int_{\mathbb R} |X'|^r dx\right)^{1/r}\;.
\end{equation}

Note that for any $r\geq 1$ we have that $\mathfrak g_{-\infty}\subset \dot W^{1,r}$,
and thus equation~\eqref{eq:W1rnorm:R} is well-defined. 
Furthermore, we remark that the Finsler norm is non-degenerate on this space, as the only constant vector fields in 
$\mathfrak g_{-\infty}$ are the zero vector fields.

The $r$-HS equation has been recently derived by Cotter et al.~\cite{cotter2020r} as the Finsler--Euler--Arnold equation on the group of diffeomorphism with respect to the Finsler metric $F$. 
As the following theorem shows, in these non-compact setting one has to pay careful attention to choose the appropriate decay conditions (c.f.~\cite{bauer2014homogeneous} for the case $r=2$). 
As a byproduct we will observe the equivalence of the $\lambda$-PJ and the $r$-HS equations.
\begin{thm}\label{thm:nonperiodicHS}
For $r\in(1,\infty)$, the geodesic equation of the right-invariant Finsler metric $F$ on the Lie group $\Diff_{-\infty}(\mathbb R)$ 
is given by
\begin{equation}\label{eq:$r$-HS_Lagrange}
r \frac{d}{dt} \left( \frac{\varphi_{tx}}{\varphi_x} \left|\frac{\varphi_{tx}}{\varphi_x}\right|^{r-2}\right)+(r-1)\left|\frac{\varphi_{tx}}{\varphi_x}\right|^r=0.
\end{equation}
The corresponding Finsler--Euler--Arnold equation --- the geodesic equation in Eulerian coordinates $u=\varphi_t\circ\varphi^{-1}$ ---
 is the $r$-Hunter--Saxton equation:
\begin{equation} \label{eq:$r$-HS_Euler}
 \left(u_x |u_x|^{r-2}\right)_{xt}+\left(u_x |u_x|^{r-2}\right)_{x}u_x+\left((u_x |u_x|^{r-2})_{x}u\right)_x=0.
\end{equation}
This equation is formally equivalent to the non-periodic $\lambda$-PJ equations 
\begin{equation}\label{gPJequation}
u_{txx} + (1+2\lambda)u_xu_{xx} +uu_{xxx} = 0,\quad \lambda = r^{-1}\in (0,1),
\end{equation}
in the sense that \eqref{gPJequation} is obtained from \eqref{eq:$r$-HS_Euler} by integrating, dividing by $|u_x|^{r-2}$ and differentiating.

On the smaller group $\Diff(\mathbb R)$  both the geodesic equation and  the Finsler--Euler--Arnold equation do not exist. 
\end{thm}
\begin{proof}
The length functional of a Finsler-metric $F$ on a (possibly infinite dimensional) manifold $\mathcal M$ is defined as
\begin{equation}
L(\varphi)=\int_0^1 F_{\varphi}(\varphi_t) dt,   
\end{equation}
where $\varphi: [0,1]\to \mathcal M$ is a path in the manifold and where $\varphi_t$ denotes its derivative.
A geodesic is a path that locally minimize the length functional; since $L$ is invariant to reparametrization, we can restrict ourselves to paths of constant speed.
By H\"older inequality, it is immediate that constant speed geodesics are exactly the local minimizers of the $q$-energy 
\begin{equation*}
E_q(\varphi)=\int_0^1 F^q_{\varphi}(\varphi_t) dt,   
\end{equation*}
for any $q> 1$.
In our case, for the $\dot W^{1,r}$-Finsler metric the most convenient choice is to consider the $q$-Energy with $q=r$. 
This leads us to the same Lagrangian as in~\cite{cotter2020r}:
\begin{equation} 
E_r(\varphi)=\int_0^1 \int_{\mathbb R}\left|\frac{\varphi_{tx}}{\varphi_x}\right|^r \varphi_x dx dt.
\end{equation}
Calculating the variation of $E_r$ in direction $\delta \varphi$ and using integration by parts one obtains
\begin{equation}\label{eq:deltaE}
\begin{split}
&\delta E_r(\varphi)(\delta \varphi)=\int_0 ^1 \int_{\mathbb R}   r\left|\frac{\varphi_{tx}}{\varphi_x}\right|^{r-2}\frac{\varphi_{tx}}{\varphi_x}\delta \varphi_{tx}-(r-1)\left|\frac{\varphi_{tx}}{\varphi_x}\right|^r  \delta \varphi_x  \ dxdt \\
&\qquad=-\int_0^1 \int_{\mathbb R}  \left( r\frac{d}{dt}\left(\left|\frac{\varphi_{tx}}{\varphi_x}\right|^{r-2}\frac{\varphi_{tx}}{\varphi_x}\right)+(r-1)\left|\frac{\varphi_{tx}}{\varphi_x}\right|^r \right) \  \delta \varphi_x  \ dxdt.
\end{split}
\end{equation}
Thus we can read off the geodesic equation:
\begin{equation}\label{eq:ELeq}
r \frac{d}{dt} \left( \frac{\varphi_{tx}}{\varphi_x} \left|\frac{\varphi_{tx}}{\varphi_x}\right|^{r-2}\right)+(r-1)\left|\frac{\varphi_{tx}}{\varphi_x}\right|^r=0.
\end{equation}
A straight-forward calculation shows that, in Eulerian coordinate $u=\varphi_t\circ\varphi^{-1}$,  this equation 
reduces to: 
\begin{equation}\label{eq:integrated-HS}
r|u_x|^{r-2}(u_{tx}+u_{xx}u)+|u_x|^r=0.
\end{equation}
Now, one obtains the $\lambda$-PJ equation~\eqref{gPJequation} by dividing by $|u_x|^{r-2}$ and differentiating in $x$, and choosing $\lambda = 1/r$.
By differentiating \eqref{eq:integrated-HS}, one gets equation~\eqref{eq:$r$-HS_Euler}, as shown in \cite[Proposition~2.2]{cotter2020r}.
%

To see that these equations do not exist on the smaller group $\Diff(\mathbb R)$ we divide equation~\eqref{eq:integrated-HS} by $|u_x|^{r-2}$ and integrate it again in the variable $x$ to obtain the formally equivalent equation
\begin{equation} \label{eq:$r$-HS_Euler_int2}
u_{t}=-uu_x+(1-\frac{1}{r})\int_{-\infty}^x u^2_x(t,z)dz\;.
\end{equation}
Note, that the constant of integration is zero, due to the decay assumptions on the the vector fields $u$. Now the non-existence follows as
for any non-trivial initial conditions  $u_0\in W^{\infty,1}(\mathbb R)$ (the Lie-algebra of $\Diff(\RR)$), 
the term $\int_{-\infty}^x (u_0)^2_x(z)dz$ dominates $u_0(u_0)_x$ for $x$ large enough.
This implies that $u_t(0,x)$ does not decay as $x\to \infty$, hence that the corresponding solution $u(t,x)\notin W^{\infty,1}(\mathbb R)$ for any $t>0$ 
(see also~\cite{bauer2014homogeneous} where this phenomenon is described for the case $r=2$). 
Alternatively, this can be seen from the solution formula as presented in Theorem~\ref{thm:geodesicsDiffR} below.
\end{proof}

\begin{rem}
In the periodic case, a similar derivation can be made, with $S^1$ instead of $\mathbb{R}$.
In this case, however $\delta \varphi_x$ is not an arbitrary function, but rather a function of zero average.
Thus, formula \eqref{eq:deltaE} implies equation \eqref{eq:ELeq} with a (non-zero) function of time on the righthand side.
The $r$-HS equation \eqref{eq:$r$-HS_Euler} follows the same way, but the $\lambda$-PJ equation \eqref{gPJequation} does not.
\end{rem}

\begin{rem}
Note that $F$ is only a weak Finsler metric, as the
the $W^{1,r}$ topology is weaker than the original $C^{\infty}$-manifold topology. 
As a consequence several results of finite dimensional Riemannian geometry do not hold in this setting. In particular, it is
not guaranteed that the geodesic distance function defines a true metric, as it can be degenerate or even vanish identically.
 As a byproduct of the analysis in the following sections, we will obtain an explicit formula for this distance function.
This will in particular imply that the geodesic distance of the $\dot W^{1,r}$-metric does not admit this misbehavior and is indeed inducing a
true distance function.
\end{rem}

\subsection{An isometry to a flat space}\label{sec:transformation_non_per}
In this section we introduce an isometry that will map the diffeomorphism group with the right-invariant $\dot W^{1,r}$-Finsler metric to an open subset of a vector space. 
As mentioned in the introduction, this construction, which is a generalization of \cite{lenells2007hunter,lenells2008hunter,bauer2014homogeneous}
for the $r=2$ case, will allow us to obtain explicit formulas for solutions to the geodesic equation on the diffeomorphism group and consequently also for the $r$-HS ($\lambda$-PJ, resp.) equation:
 \begin{thm}\label{thm:isometry:nonperiodic}
For $r\in [1,\infty)$, the mapping
\begin{equation}
\Phi: \begin{cases}
\left(\Diff_{-\infty}(\mathbb R), F\right)  	&\to  \left(W^{\infty,1}(\mathbb R),L^r\right)\\ 
\varphi 						&\mapsto r\left(\varphi_x^{\frac{1}{r}}-1\right)
\end{cases}
\end{equation}
is an isometric embedding. Furthermore, the image $\mathcal U=\Phi(\Diff_{-\infty}(\mathbb R))$  is 
the set of all positive functions in $W^{\infty,1}(\mathbb R)$, i.e.,
\begin{equation}\label{eq:nonperiodic:U}
\mathcal U=\{f \in W^{\infty,1}(\mathbb R):f>-r\}.
\end{equation}
The inverse of $\Phi$ is given by
\begin{equation}\label{eq:Phiinverse:nonperiodic}
\Phi^{-1}: \begin{cases}
W^{\infty,1}(\mathbb R)  &\to  \Diff_{-\infty}(\mathbb R) \\ f &\mapsto x+\int_{-\infty}^x \left( \left(\frac{f(\tilde x)}{r}+1\right)^r-1\right)d\tilde x.
\end{cases}
\end{equation}
 \end{thm}

\begin{proof}
First, note that since $\varphi_x -1 \in W^{\infty,1}(\RR)$ and non-negative, a straightforward calculation shows that $\varphi_x^{1/r}-1$ is also in $W^{\infty,1}(\RR)$ (since $(1+\alpha)^{1/r} < 1 + \frac{1}{r}\alpha$).
Similarly, it is easy to see that the image of the inverse is indeed in $\Diff_{-\infty}(\RR)$.
Next, we calculate the variation formula of the mapping $\Phi$. We have:
\begin{align}
D_{\varphi,h}\Phi= {\varphi_x}^{\frac{1}{r}-1} h_x
\end{align}
and thus
\begin{align}
\|D_{\varphi,h}\Phi\|^r_{L^r}= \int_{\mathbb R} |{\varphi_x}^{\frac{1}{r}-1} h_x|^r dx=
\int_{\mathbb R} {\varphi_x}^{1-r} |h_x|^r dx=F_{\varphi}(h)^r.
\end{align}
It remains to prove the statement on the image.  Let $f= r({\varphi_x}^{\frac{1}{r}}-1)=\Phi(\varphi)$ for some 
$\varphi\in \Diff_{-\infty}(\mathbb R)$. Since elements of $\Diff_{-\infty}(\mathbb R)$
are orientation preserving diffeomorphisms, this implies that $\varphi_x>0$. Thus it follows that $f>-r$, which concludes the characterization of the image. The statement on the inverse follows by direct calculation. 
\end{proof}

\subsection{A solution formula for the $r$-HS ($\lambda$-PJ) equation}\label{sec:sol_formula}
As a consequence of Theorem~\ref{thm:isometry:nonperiodic} and the simple form of the image of $\Phi$ we obtain an explicit formula for geodesics on $\left( \Diff_{-\infty}(\mathbb R), F\right)$.
\begin{thm}\label{thm:geodesicsDiffR}
Let $r\in(1,\infty)$. Given initial conditions 
$$\varphi(0) = \operatorname{id}\in \Diff_{-\infty}(\mathbb R),\quad \varphi_t(0)=u_0\in  T_{\operatorname{id}}\Diff_{-\infty}(\mathbb R)=\mathfrak g_{-\infty}$$ 
the unique solution to the geodesic equation of $F$ is given by:
\begin{equation}\label{eq:geodesicDiffnonperiodic}
\varphi(t,x)=x + \int_{-\infty}^x \left(\left(\frac{tu_0'(y)}{r}+1\right)^r - 1 \right) dy.
\end{equation}
Consequently, the solution to the $r$-HS (equivalently, $\lambda$-PJ with $\lambda= r^{-1}$) equation with initial condition $u(0)=u_0\in \mathfrak g_{-\infty}$
is given by $u=\varphi(t,\varphi^{-1}(t,x))$ with $\varphi$ given by~\eqref{eq:geodesicDiffnonperiodic}.

In particular, this implies that the equation is well-defined on $\Diff_{-\infty}(\mathbb R)$, and that these solutions are length minimizing paths with respect to the distance function induced by the $\dot W^{1,r}$-metic.
\end{thm}

\begin{rem}
As we will discuss in Section~\ref{sec:r<1} below, formula~\eqref{eq:geodesicDiffnonperiodic}, with $r=\lambda^{-1}$, provides a solution to the $\lambda$-PJ equation for every $\lambda\ne 0$; however for $\lambda\notin(0,1)$ we do not have an interpretation of the equation as a geodesic equation.
The case $\lambda = 0$, which is equivalent to $r=\infty$, is treated in Section~\ref{sec:r_to_infty}.
\end{rem}

\begin{proof}
By Theorem~\ref{thm:isometry:nonperiodic} the mapping $\Phi$ is an isometry from $\left( \Diff_{-\infty}(\mathbb R), F\right)$ to an open subset of the vector space $(W^{\infty,1}(\mathbb R),L^r)$. Thus geodesics on the former space 
are the pre-images of $\Phi$ of geodesics in the image, i.e., pre-images of straight lines and thus the formula follows directly from the inversion formula~\eqref{eq:Phiinverse:nonperiodic}.

Since straight lines are length minimizing in $(W^{\infty,1}(\mathbb R),L^r)$, so are their pre-images in $\Diff_{-\infty}(\mathbb R)$.
\end{proof}

\subsection{Geodesic incompleteness and blowup of the  $r$-HS ($\lambda$-PJ) equation}\label{sec:blowup}
As a direct consequence of  the geometric interpretation in Theorem~\ref{thm:geodesicsDiffR} we obtain the following result concerning
geodesic incompleteness of $\left(\Diff_{-\infty}(\mathbb R),F\right)$. This, in turn, implies blow-up for the $r$-HS ($\lambda$-PJ resp.) equation.
The following theorem is in correspondence with the blow-up result in~\cite[Theorem 3.3]{cotter2020r}.
\begin{cor}
Let $r\in(1,\infty)$. The space $\left(\Diff_{-\infty}(\mathbb R),F\right)$ is geodesically incomplete. More precisely, given any initial conditions 
$$\varphi(0) = \operatorname{id}\in\Diff_{-\infty}(\mathbb R),\quad \varphi_t(0)=u_0\in 
 T_{\operatorname{id}}\Diff_{-\infty}(\mathbb R)=\mathfrak g_{-\infty}$$ 
the geodesic as given by formula~\eqref{eq:geodesicDiffnonperiodic} exists for all time $t>0$ if and only if $u_0'(x)\geq 0$ for all $x\in \mathbb R$. 
If there exists a point $x_0$ with $u_0'(x_0)< 0$ then the geodesic only exists  for  finite time $T^*(u_0)$, with
\begin{equation}\label{eq:blowup}
T^{*}(u_0)
= -\frac{r}{\inf_{x\in \RR} u_0'(x)}
\end{equation}

At time $T^{*}(u_0)$ we have $\min_{x\in S^{1}}\varphi_x(T^{*}(u_0),x)=0$ and thus the corresponding solution $u=\varphi_t\circ\varphi^{-1}$ to the $r$-HS (equivalently, $\lambda$-PJ with $\lambda=r^{-1}$) equation  blows up at the same time.
\end{cor}

\begin{rem}
Note that this shows that every geodesics blows up either in finite positive or finite negative time.  
Furthermore, we see that the maximal existence time of geodesics 
grows linearly with the parameter $r$.
\end{rem}

\begin{proof}
To obtain this result, we only need to observe that a geodesics $\varphi(t,x)$ in $\left(\Diff_{-\infty}(\mathbb R),F\right)$ ceases to exists if and only if $\varphi_x(t,x)$ approaches zero (as there is no loss of regularity).
By the solution formula~\eqref{eq:geodesicDiffnonperiodic} this is equivalent to:
\begin{equation}
\left(\frac{tu_0'(x)}{r}+1\right)^r  =0,
\end{equation}
which directly leads to the desired statement.
Note that $\varphi_x(t,x)=0$
implies that $u=\varphi_t\circ \varphi^{-1}$ blows up.
\end{proof}

\subsection{The metric completion}
In this section we will calculate the metric completion. 
Note that the solution formula~\eqref{eq:geodesicDiffnonperiodic} is no longer well-defined on this larger space, as $u\in AC(\mathbb R)$ does not imply that $u'\in L^r(\mathbb R)$. 
We will study a slightly smaller functions space such that the solution formula is still well-defined later in Section~\ref{sec:lowregularity_nonperiodic}.

\begin{prop}
Let $r\in(1,\infty)$. The metric completion of $\left(\Diff_{-\infty}(\mathbb R),\operatorname{dist}^F\right)$ is the monoid of all absolutely contionous surjective maps:
\begin{equation*}
\operatorname{Mon}(\mathbb R)=\left\{\varphi\in AC(\mathbb R): \varphi \text{ is surj., }\varphi'\geq 0 \text{ a.e. and}  \lim_{x\to -\infty} (\varphi(x)-x)=0   \right\}.
\end{equation*}
\end{prop}
\begin{proof}
To calculate the metric completion of $\left(\Diff_{-\infty}(\mathbb R),\operatorname{dist}^F\right)$ it is sufficient to study the metric completion of the image $\mathcal U$ of $\Phi$, i.e.,
the $L^r$ closure of $\mathcal U$, which is given by
\begin{equation}
\overline{\mathcal U}= \left\{ f\in L^{r}(\mathbb R): f\geq -r \text{ a.e.} \right\}.
\end{equation} 
Using the inversion formula for $\Phi$ it is easy to show that the elements in the pre-image of $\Phi$ are precisely the elements 
of $\operatorname{Mon}(S^1)$.
\end{proof}

\subsection{Geodesic convexity}\label{sec:bvp}
In the previous sections we have seen that the initial value problem is locally well-posed, but that solutions might blow-up in finite time. 
In this section we show the geodesic boundary value problem is better behaved. 
Namely, we show that for any any boundary conditions $\varphi_0,\varphi_1\in \Diff_{-\infty}(\mathbb R)$ there exists a minimizing geodesic. 
This is, again, a consequence of the fact that our space is isometric to a convex open subset of a vector space.
\begin{cor}
For every $r\in(1,\infty)$, the space $\left(\Diff_{-\infty}(\mathbb R),F\right)$ is geodesically convex. More precisely, given any boundary conditions 
$$\varphi(0) = \operatorname{id}\in\Diff_{-\infty}(\mathbb R),\quad \varphi(1)=\varphi_1\in\Diff_{-\infty}(\mathbb R)$$ 
there exists a unique minimizing geodesic connecting $\operatorname{id}$ to $\varphi_1$. 
\end{cor}
\begin{proof}
This result follows immediately from the convexity of the set $\mathcal U$ as a subset of an infinite dimensional vector space.
\end{proof}

\subsection{Extensions to low regularity}\label{sec:lowregularity_nonperiodic}
The analysis above holds for spaces of lower regularity: 
all the constructions hold without any major change on any space of diffeomorphisms of $\RR$ for which the following holds: (i) it is closed under taking inverse; (ii) composition from the right is smooth; (iii) the $\dot{W}^{1,r}$ Finsler norms make sense on it. 
A natural space to consider in this case, which includes $\Diff_{-\infty}(\RR)$, is the space of integrable bi-Lipschitz homeomorphisms:
\begin{multline}
\biLip^{1,1}_{-\infty}(\RR) := \Big\{ \varphi=\id+f:   \varphi \text{ is invertible, }  \varphi^{-1}=\id+g,\\ f,g \in \dot{W}^{1,1}(\RR)\cap W^{1,\infty}(\RR), \,\, \lim_{x\to -\infty} f(x) = 0\Big\}.
\end{multline}
Here $\dot W^{1,1}(\RR)$ is the space of functions with an integrable derivative. 
This space is a manifold and a topological group, in which composition from the right is smooth (thus it is a \emph{half Lie-group} in the terminology of \cite{kriegl2015exotic,marquis2018half}).
Its Lie-algebra is the space
\[
\mathfrak{lip}^{1,1}_{-\infty}(\RR) := \{ u \in W^{1,\infty}(\RR)\cap \dot{W}^{1,1}(\RR) ~:~ \lim_{x\to -\infty} u(x) =0 \},
\]
on which all the $\dot{W}^{1,r}$ Finsler norms are well defined and also the $r\to \infty$ limit make sense (see below).
Moreover, the basic mapping and the solution formula from Theorem~\ref{thm:geodesicsDiffR} naturally extend to this space; extensions 
to any larger space of functions seem difficult. 

The space $\biLip^{1,1}_{-\infty}(\RR)$ includes, in particular, all bi-Lipschitz piecewise-linear homeomorphisms of $\RR$ that decay at $-\infty$:
\[
\Diff^{\text{PL}}_{-\infty}(\RR) := \left\{ \varphi=\id+f \in \biLip^{1,1}_{-\infty}(\RR) ~:~ f' \text{ is piecewise constant } \right\}.
\]
Using the solution map \eqref{eq:geodesicDiffnonperiodic} we immediately obtain the following result concerning this submanifold:
\begin{cor}
For every $r\in(1,\infty)$, the space $\Diff^{\text{PL}}_{-\infty}(\RR)$ is totally geodesic in $\biLip^{1,1}_{-\infty}(\RR)$ with respect to the flow of the $r$-Hunter--Saxton equation.
That is, any solution whose initial velocity $u_0\in \mathfrak{lip}^{1,1}_{-\infty}(\RR)$ is piecewise-linear, remains in $\Diff^{\text{PL}}_{-\infty}(\RR)$ as long as the flow exists. 
\end{cor}

These piecewise linear solutions were studied in detail in \cite{cotter2020r}, and this corollary provides a geometric framework for them.

\subsection{The limit $r\to \infty$.}\label{sec:r_to_infty}
In the following we study the limiting behavior of $r$-HS  as $r\to \infty$.
In particular, we show that the limiting solutions are solutions of the $\lambda$-PJ equation for $\lambda=0$,
and that the $\dot{W}^{1,\infty}$-norm of the velocity is preserved along the flow, thus answering several questions posed by Cotter et al.~\cite{cotter2020r}.

\begin{thm}[Lagrangian viewpoint]
Let $\vp^r(t,x)$ be the solution of the $r$-Hunter--Saxton equation with initial conditions $\vp^r(0,x) = x$ and $\vp^r_t(0,x) = u_0 \in \frakg_{-\infty}$.
Then we have the pointwise limit
\begin{equation}\label{eq:infty_geodesic_nonperiodic}
\lim_{r\to \infty} \vp^r(t,x) = x + \int_{-\infty}^x \left(e^{tu_0'(y)}-1\right) \,dy =: \varphi^\infty(t,x),
\end{equation}
which exists for all time.
The function $\varphi^\infty(t,x)\in \Diff_{-\infty}(\RR)$ satisfies the differential equation
\begin{equation}\label{eq:inftyHS_Lagrange}
\left(\frac{\varphi_{tx}}{\varphi_x} \right)_t = 0,
\end{equation}
which is the formal limit of equation \eqref{eq:$r$-HS_Lagrange} as $r\to \infty$.
\label{thm:inftyRS_Lagrange}
\end{thm}

\begin{cor}[Eulerian viewpoint] \label{cor:inftyRS_Euler}
The vector field associated with \eqref{eq:infty_geodesic_nonperiodic} satisfies the equation 
\begin{equation}\label{eq:inftyHS_Euler}
u_{xt} + u_{xx} u = 0,
\end{equation}
which is the Eulerian version of \eqref{eq:inftyHS_Lagrange}.
By differentiating, this equation is equivalent to the $\lambda$-PJ equation with $\lambda = 0$.

The derivative of the vector field satisfies
\begin{equation}\label{eq:uinfty_x}
u^\infty_x(t,x) = u_0'((\vp^\infty)^{-1}(t,x)),
\end{equation}
and thus the $\dot{W}^{1,\infty}$-norm of $u$ is preserved along the flow.
\end{cor}

\begin{proof}[Proof of Theorem~\ref{thm:inftyRS_Lagrange}]
Taking the limit $r\to \infty$ of $\vp^r$ as given by the solution formula \eqref{eq:geodesicDiffnonperiodic}, we immediately obtain \eqref{eq:infty_geodesic_nonperiodic} by the monotone convergence theorem.
A straight-forward calculation shows that \eqref{eq:infty_geodesic_nonperiodic} satisfies the equation \eqref{eq:inftyHS_Lagrange}.

Writing \eqref{eq:$r$-HS_Lagrange} as 
\[
\left(\frac{1}{r-2}\left|\frac{\varphi_{tx}}{\varphi_x}\right| + \frac{\varphi_{tx}}{\varphi_x}\right) \left(\frac{\varphi_{tx}}{\varphi_x} \right)_t +\frac{r-1}{r(r-2)}\left|\frac{\varphi_{tx}}{\varphi_x}\right|^3=0,
\]
and taking the formal limit $r\to \infty$, we obtain \eqref{eq:inftyHS_Lagrange} (similarly, \eqref{eq:inftyHS_Euler} is formally obtained from \eqref{eq:$r$-HS_Euler_int2} by differentiating and taking $r\to \infty$).
\end{proof}

\begin{proof}[Proof of Corollary~\ref{cor:inftyRS_Euler}]
From \eqref{eq:infty_geodesic_nonperiodic} we have that $\vp^\infty_x (t,x) = e^{tu_0'(x)}$ and $\vp_{xt}^\infty(t,x) = u_0'(x) e^{tu_0'(x)}$, and thus
\[
u^\infty_x(t,\vp(t,x)) = \frac{\vp^\infty_{xt}(t,x)}{\vp^\infty_x(t,x)} = u_0'(x),
\]
from which \eqref{eq:uinfty_x} follows, and thus
\[
\|u^\infty_x(t,\cdot)\|_\infty = \|u_0'\|_\infty,
\]
which proves the preservation of the $\dot{W}^{1,\infty}$-norm.
Differentiating \eqref{eq:uinfty_x} and using \eqref{eq:inftyHS_Lagrange} immediately shows that $u^\infty$ satisfies \eqref{eq:inftyHS_Euler}.
\end{proof}

\subsection{$r= 1$ and Burgers' equation}
In \cite[Remark~3.4]{cotter2020r} it was observed that the $1$-HS equation is formally equivalent  to the inviscid Burgers' equation. 
In the following we will give a geometric interpretation for this phenomenon.
The inviscid Burgers' equation
\begin{equation*}\label{eq:burger}
u_t+uu_x=0,
\end{equation*}
in Lagrangian coordinates, is transformed to the straight line equation $\varphi_{tt}=0$, where $\varphi_t=u\circ\varphi$.
Thus, the Burgers' equation can be interpreted as (is equivalent to) the geodesic equation of any Finsler metric on $\Diff(\mathbb R)$, where geodesics are given by straight lines.
This happens, as is widely known, for the flat $L^2$ metric, but also, in fact, for any flat (non-invariant) $W^{k,r}$-Finsler metric,
\begin{equation*}
\bar F_{\varphi}(h) = \|h\|_{W^{k,r}},
\end{equation*}
as well as their homogeneous counterparts.
Here, by flat we mean that the norm of $h$ is independent of the footpoint $\varphi$, unlike invariant metrics which are the focus of this paper, which correspond to taking the appropriate norm of $h\circ \vp^{-1}$.

However, for the specific case of the homogeneous $W^{1,1}$-Finsler metric, the flat (non-invariant) metric and the invariant metric coincide:
\begin{align*}
F_{\varphi}(h)=\int_{\mathbb R} |(h\circ\varphi^{-1})'| dx=
\int_{\mathbb R} \left|\left(\tfrac{h'}{\varphi'}\right)\circ\varphi^{-1}\right| dx
=
\int_{\mathbb R} |h'| dx=\bar F_{\varphi}(h),
\end{align*}
and thus the geodesic equation of the $\dot W^{1,1}$-Finsler metric, which formally corresponds to the $1$-HS equation, is equivalent to the Burgers' equation.
Note that this is the only value for $r$ and $k$ such, that the non-invariant and invariant $W^{k,r}$-metrics are equal.

\subsection{The range $r<1$}\label{sec:r<1}
In this section we discuss the parameter range $r<1$ (which corresponds to $\lambda\in (-\infty,0)\cup (1,\infty)$). This is of particular importance as it includes the original PJ equation --- $r=\lambda=-1$ --- and is related to self-similar axisymmetric Navier--Stokes equation in higher dimensions. The space $L^r$ is only a normed space for $r\geq1$ and thus $F$ does not define a Finsler structure if $r<1$. Consequently we do not have a geometric interpretation of the $\lambda$-PJ ($r$-HS, resp.) equation as a geodesic equation of a Finsler metric on a diffeomorphism group. However, much of our analysis, including the explicit solution formula does go through, at least formally.

First, we note that for any $\lambda\in \mathbb{R}$, the $\lambda$-PJ equation, in Lagrangian coordinates, can be written as
\begin{equation}\label{eq:lambdaPJLagrange}
\frac{d}{dt}\left(\frac{\varphi_{xt}}{\varphi_{x}}\right) + \lambda \left(\frac{\varphi_{xt}}{\varphi_{x}}\right)^2 = 0,
\end{equation}
which is equivalent to \eqref{eq:$r$-HS_Lagrange} (with $r=\lambda^{-1}$) as long as $r\ne 0,1$. 
Thus, the proof of Theorem~\ref{thm:nonperiodicHS} shows, that at least formally (without worrying about integrability or differentiability of the functions involved), the $\lambda$-PJ equation is equivalent to the Euler-Lagrange equation of the invariant energy
\[
E_{\lambda^{-1}}(\varphi) = \int_{0}^1 \int_\mathbb{R} \left| \frac{\varphi_{xt}}{\varphi_{x}}\right|^{1/\lambda} \varphi_x \,dx\,dt
\]
for any $\lambda \ne 0,1$, and not only for $\lambda \in (0,1)$ as in Theorem~\ref{thm:nonperiodicHS}.
However, for $\lambda\notin (0,1)$, it is not clear how to relate this energy as a geodesic equation, and if $\lambda<0$ (equiv.\ $r<0$), then there are issues of convergence of this integral.
We note that these are the only invariant energies of the form
\[
 \int_{0}^1 \int_\mathbb{R} f\left( \frac{\varphi_{xt}}{\varphi_{x}}\right) \varphi_x \,dx\,dt,
\]
that yield the $\lambda$-PJ equation, as a straightforward calculation shows that $f$ must satisfy
\[
\lambda s^2 f''(s) - s f'(s) + f(s) = 0, \qquad f''\not\equiv 0,
\]
the solutions of which are $f(s) = s^{1/\lambda}$ as above (in the case $\lambda=r=1$ we also get $f(s) = s\log s$).

Finally, a straightforward calculation shows that \eqref{eq:geodesicDiffnonperiodic} indeed solves \eqref{eq:lambdaPJLagrange}, for $0\ne r=\lambda^{-1}$, which leads to the following corollary:
\begin{cor}
Let $r\in \mathbb R\setminus\{0\}$. Given initial conditions 
$u_0\in \mathfrak g_{-\infty}$ the unique solution to
the $r$-HS equation ($r^{-1}$-PJ, resp.) is given by
$u=\varphi(t,\varphi^{-1}(t,x))$, where $\varphi$ is given by
\eqref{eq:geodesicDiffnonperiodic}. 
\end{cor}

Note that formula \eqref{eq:blowup} for the blowup time still holds whenever $r>0$. 
For $r<0$ (that is $\lambda <0$), there is blowup of $\varphi$ if and only if $u_0'(x_0)>0$ for some $x_0$  (the exact converse of the case $r>0$), and in this case the blow-up time is given by $T^*(u_0) = \frac{|r|}{\sup_{x\in \RR} u_0'(x)}$. 
The nature of the blowup is also different between $r>0$ and $r<0$: in the former, $\varphi$ loses the immersion property at the blowup time (that is, we have $\varphi_x(T^*,x_0) = 0$ at some point);
in the latter, $\varphi_x$ blows up at some point at the blowup time.
For a detailed study of blowup in this range for the periodic $\lambda$-PJ equation, see Sarria and Saxton~\cite{SS13} and the references therein (note that their parameter $\lambda$ corresponds to our $-\lambda$).

\section{The periodic case}\label{sec:periodic}
In this section we briefly discuss the periodic situation.
From a functional analytic point of view the compactness of the domain simplifies the situation --- there are no decay conditions required to equip
the diffeomorphism group with a manifold structure. 
The geometric picture, however, is more complicated: we will show that the diffeomorphism group with the $\dot W^{1,r}$ metric is isometric to an open subset of an $L^r$-sphere.

The derivation of the equations themselves is similar to the non-periodic case, as well as the generalization to lower regularity, and thus we do not repeat them here.

\subsection{Diffeomorphisms on the circle and the $r$-HS equation}
We start by introducing the group of smooth, orientation preserving diffeomorphims on the circle, i.e., we consider the space
\begin{equation}
\Diff(S^1)=\left\{\varphi\in C^{\infty}(S^1,S^1) ~:~  \varphi'>0, \,\,\varphi^{-1}\in  C^{\infty}(S^1,S^1) \right\}.
\end{equation}
It is well known (see, e.g., \cite{omori1970group}) that the space $\Diff(S^1)$ is a smooth, infinite dimensional Fr\'echet Lie-group with Lie-algebra 
$\mathfrak g=C^{\infty}(S^1)$ the space of vector fields on $S^1$.
The right-invariant $\dot W^{1,r}$-Finsler metric, as introduced in~\eqref{eq:W1rnorm:R} with integration over $\mathbb R$ replaced 
by integration over $S^1$, is only a degenerate Finsler metric on this group (constant vector fields are in the kernel).
This  leads us to consider the metric on the homogenous space 
$\on{Rot}\backslash \Diff(S^1)$ of Sobolev diffeomorphisms modulo rotations, which we will identify 
with the section 
\begin{equation}
\Diff_0(S^1)=\left\{\varphi\in\Diff(S^1): \varphi(0)=0\right\},
\end{equation}
where we identified the circle $S^1$ with the interval $[0,1]$.
Note, that Lie-Algebra of $\Diff_0(S^1)$ consists of the space
$$C^{\infty}_0(S^1):=\left\{u\in C^{\infty}(S^1): u(0)=0\right\}.$$
As the kernel of the norm $F$  consists exactly of all constant vector fields,
It is now easy to see that the norm $F$ as defined in~\eqref{eq:W1rnorm:R} defines a right-invariant Finsler metric on $\Diff_0(S^1)$
(recall that the only constant tangent vector to
$\Diff_0(S^1)$ is the zero vector field).
The periodic $r$-HS equation can now be interpreted as the Finsler--Euler--Arnold equation on $\Diff_0(S^1)$, i.e., the analogue of the existence part of Theorem~\ref{thm:nonperiodicHS}
holds in the periodic case.

Also, similarly to Section~\ref{sec:lowregularity_nonperiodic}, this formulation allows us to consider these equations in lower regularity, namely on the space of bi-Lipschitz, orientation preserving homeomorphisms (modulo rotations):
\[
\biLip_0(S^1) := \Big\{ \varphi \in W^{1,\infty}(S^1) :   \varphi'>0 ,\,  \varphi^{-1}\in W^{1,\infty}(S^1),\, \varphi(0)=0 \Big\},
\]
whose Lie-algebra is the space of Lipschitz velocities
\[
\mathfrak{lip}_0(S^1) := \{ u \in W^{1,\infty}(S^1) ~:~ u(0)=0 \}.
\]

\subsection{An isometry to the sphere}\label{sec:transformation_per}
We will now construct an isometry that will map the 
diffeomorphism group with the right-invariant $\dot W^{1,r}$-Finsler metric to an infinite dimensional $L^r$-sphere of radius $r$.
As mentioned earlier, this transformation was already used by Sarria and Saxton to study the $\lambda$-PJ equation \cite{SS13}.\footnote{It was also used as a tool for studying the diameter of diffeomorphism groups \cite{bauer2019can}.}
The merit of this section is in revealing the geometric meaning of it in the context of the $r$-HS equation.
Also, we hope that this viewpoint will provide further tools to study these equations by studying the $L^r$-sphere (in a similar way to the Hunter--Saxton case \cite{lenells2007hunter}). 

 \begin{thm}\label{thm:isometry:periodic}
Let $r\in[1,\infty)$. The mapping
\begin{equation}
\Phi: \begin{cases}
\left(\Diff_0(S^1), F\right)  &\to  \left(C^{\infty}({S^1}),L^r\right)\\ \varphi &\mapsto r\;{\varphi_x}^{\frac{1}{r}}
\end{cases}
\end{equation}
is an isometric embedding. Furthermore, the image $\mathcal U=\Phi(\Diff_0(S^1))$  is an open subset of the $L^r$-sphere of radius $r$ given by 
\begin{equation}
\mathcal U=\{f \in C^{\infty}({S^1}):f>0, \|f\|_{L^r}=r\}.
\end{equation}
The inverse of $\Phi$ is given by
\begin{equation}\label{eq:Phiinverse}
\Phi^{-1}: \begin{cases}
\mathcal U  &\to  \Diff_0(S^1) \\ f &\mapsto  \frac{1}{r^r} \int_0^x |f(\tilde x)|^r d\tilde x.
\end{cases}
\end{equation}
 \end{thm}
  
\begin{proof}
The variation formula can be obtained exactly as in the non-periodic case. Thus it only remains to prove the statement on the image. Therefore let $f= r\;{\varphi_x}^{\frac{1}{r}}=\Phi(\varphi)$ for some 
$\varphi\in \Diff_0(S^1)$. Since elements of $\Diff_0(S^1)$
are orientation preserving diffeomorphisms it follows that $f>0$. Furthermore, by identifying $S^1$ with the interval $[0,1]$, we have
\begin{align}
1=\varphi(1)-\varphi(0)=\int_0^1 \varphi_x\;  dx=    \frac{1}{r^r}\int_0^1 f^r dx = \frac{1}{r^r} \|f\|^r_{L^r},
\end{align}
which concludes the characterization of the image. The statement on the inverse follows by direct calculation. 
\end{proof}

\bibliographystyle{abbrv}

\end{document}